\documentclass[a4paper,reqno]{amsart}
\usepackage{mathrsfs,amsmath,amsfonts,amsthm,amssymb,graphicx,pdfpages,lipsum}
\usepackage{hyperref}
\usepackage[all,2cell]{xy}\UseAllTwocells \SilentMatrices

\newcommand{\ca}{\mathcal} 
\newcommand{\Hom}{\ensuremath{\mathrm{Hom}}}
\newcommand{\Comod}{\ensuremath{\mathbf{Comod}}}
\newcommand{\Mod}{\ensuremath{\mathbf{Mod}}}
\newcommand{\Alg}{\ensuremath{\mathbf{Alg}}}
\newcommand{\Coalg}{\ensuremath{\mathbf{Coalg}}}
\DeclareMathOperator*\colim{colim}
\numberwithin{equation}{section}

\theoremstyle{plain}
\newtheorem{prop}{Proposition}[section]
\theoremstyle{remark}
\newtheorem{rmk}{Remark}[section]
\theoremstyle{plain}
\newtheorem{lem}{Lemma}[section]	
\theoremstyle{plain}
\newtheorem{cor}{Corollary}[section]
\theoremstyle{plain}
\newtheorem{thm}{Theorem}[section]

\begin{document} 

\title{Enrichment of Categories of Algebras
and Modules}
\author{Christina Vasilakopoulou}
\email{C.Vasilakopoulou@dpmms.cam.ac.uk}

\maketitle

\vspace{-4ex}
\begin{center}
\emph{\footnotesize Department of Pure Mathematics and Mathematical Statistics\\
University of Cambridge}
\end{center}

\begin{abstract} 

We study the universal measuring coalgebras $P(A,B)$ of Sweedler and the
universal measuring comodules $Q(M,N)$ of Batchelor. We show that these
universal objects exist in a very general context. We provide a detailed 
proof of an observation of Wraith, that the $P(A,B)$ are the 
hom-objects of an enrichment of algebras in coalgebras. 
We show also that the $Q(M,N)$ provide an enrichment of the global 
category of modules in the global category of comodules.

\end{abstract}

\section{Introduction}\label{intro}

The notion of the universal measuring coalgebra $P(A,B)$ for algebras $A$ 
and $B$ was first introduced by Sweedler in \cite{Sweedler}. 
The elements of $P(A,B)$ can be thought of
as generalized maps from $A$ to $B$. Examples of this point of view are 
given by Marjorie Batchelor in \cite{MR1793011}. In the early 1970's,
Gavin Wraith suggested that the coalgebra $P(A,B)$ can be used to give
an enrichment of the category of algebras in the category of coalgebras. In this 
paper, we work through this fact in detail, since there is no treatment
of Wraith's idea in the literature. We also extend his idea to
give an enrichment of the global category of modules $\Mod$ in the 
global category of comodules $\Comod$, using the measuring comodules, 
introduced by Batchelor in \cite{MR1793011}.

This paper is motivated by an issue concerning the two classic notions
of enrichment and fibrations. There is no evident notion of 
fibration in the enriched setting. Here we consider the well-known fibration
of the global category of modules $\Mod$ over the category of $R$-algebras $\Alg_R$
for $R$ a commutative ring, and alongside it the opfibration of the category of 
comodules $\Comod$ over the category of $R$-coalgebras $\Coalg_R$. 
Since $\Alg_R$ and $\Mod$ are enriched in $\Coalg_R$ and $\Comod$ respectively, 
we would like to see this situation as some kind of \emph{enriched fibration}.  

Section \ref{background} gives the background for the development of the following sections.
The basic facts about categories of monoids/comonoids and 
modules/comodules in monoidal categories are presented, with particular
emphasis on the symmetric monoidal category $\Mod_R$ for a commutative ring $R$. 
Additionally, we discuss locally presentable categories which are a 
useful context in which to frame later constructions. 
General references on these subjects for our 
purposes are \cite{Sweedler}, \cite{Quantum}, \cite{HopfAlg} and \cite{MR1294136}.  
Finally, we recall parts of the theory of the action of a monoidal category 
$\ca{V}$ on an ordinary category $\ca{D}$, 
which lead to an enrichment of $\ca{D}$ in $\ca{V}$,
as in \cite{MR1897810}. This is actually a special case 
of the more general result discussed in \cite{enrthrvar}, 
that there is an equivalence 
between the 2-category of tensored $\ca{W}$-categories and 
the 2-category of $\ca{W}$-representations, for $\ca{W}$ a 
right-closed bicategory.

In Section \ref{existmeascoal}, we consider the existence of 
the universal measuring coalgebra. 
The question that motivated the definition 
of measuring coalgebras is under which 
conditions, for $A,B$ $k$-algebras and $C$ a 
$k$-coalgebra ($k$ a field), the linear map 
$\rho\in\Hom(A,\Hom(C,B))$ corresponding 
under the usual tensor-hom adjunction to 
$\sigma\in \Hom(C\otimes A,B)$ in $\mathbf{Vect_k}$,
is actually an algebra map. In \cite{Sweedler}, 
Sweedler defines what it means for 
a linear map $C\otimes A\xrightarrow{\sigma}B$ to 
\emph{measure}, and so gives a category 
of measuring coalgebras $(\sigma,C)$; he also gives a 
concrete construction of a terminal object $P(A,B)$ 
in this category, 
defined by the natural bijections
\begin{displaymath}
\Alg(A,\Hom(C,B))\cong
\{\sigma\in\Hom(C\otimes A,B)|\sigma 
\,\ \mathrm{measures}\}\cong
\Coalg(C,P(A,B)),
\end{displaymath} 
called the \emph{universal measuring
coalgebra}. We identify the more 
general categorical ideas underlying this development,
and prove the existence of the object $P(A,B)$ in a
broader context. In particular, our construction covers the
case of $\Mod_R$ for 
a commutative ring $R$. The class of 
\emph{admissible} monoidal categories and results from 
\cite{MR2500049} and \cite{MR2529129} play an important role
in this process, and in particular the 
natural isomorphism defining the object $P(A,B)$ in 
$\Coalg_R$ is also provided by \cite[Proposition 4]{MR2500049}.  

In Section \ref{enrichmalgscoalgs}, we combine the results of the two previous sections
and we prove in our general context the result of Wraith,
that there is an enrichment of the ordinary category $\Alg_R$ 
in the category $\Coalg_R$, with hom-objects $P(A,B)$.

In Section \ref{catsModComod}, we define the \emph{global category of modules} 
and \emph{comodules},
$\Mod$ and $\Comod$. These two categories arise via the 
Grothendieck construction, corresponding
to the functors sending an algebra (respectively coalgebra) 
to the category 
of its modules (respectively comodules) in $\Mod_R$, and they have nice 
categorical properties. In particular, as observed by 
Wischnewsky at the end of \cite{MR0414567}, 
the category $\Comod$ is comonadic over the simple product category 
$\Mod_R\times\Coalg_R$.

In Section \ref{existmeascomod}, the universal measuring comodule $Q(M,N)$ is defined, via a natural 
isomorphism similar to the one defining the universal measuring algebra,
$\Comod(X,Q(M,N))\cong\Mod(M,\Hom(X,N))$. Applications of measuring 
comodules (as well as measuring coalgebras) can be found in \cite{MR1793011}. 
We give a proof of existence of $Q(M,N)$ again in our general setting.
This proof is not a direct generalization of the proof of
existence of $P(A,B)$. Rather it makes detailed use
of the internal structure of $\Comod$ and $\Mod$. 

In Section \ref{enrichmmodscomods}, in an analogous way to 
Section \ref{enrichmalgscoalgs}, we show how the global category 
$\Mod$ is enriched in the symmetric monoidal closed category $\Comod$, with hom-objects
$Q(M,N)$. Hence, at this point we have established the situation where the 
domain and codomain of the fibration $\Mod\to\Alg_R$ are enriched over
the domain and codomain of the opfibration $\Comod\to\Coalg_R$.

In Section \ref{catComodrevisit}, we additionally prove that $\Comod$ is monoidal closed, and we
sketch how, when we work in $\mathbf{Vect}_k$ for a field $k$, the existence 
of the measuring comodule is more straightforward than in the general case.

I would like to thank Prof. Martin Hyland 
for posing the driving questions for this project, and
also for his invaluable suggestions and advice.
Furthemore, I would like to thank Dr. Ignacio
Lopez Franco for his significant contribution
on matters of direction and methodology.

\section{Background}\label{background}

\subsection{Categories of Monoids and Comonoids}\label{catsmonscomons} 
The categories of monoids and
comonoids for a monoidal category $\ca{V}$ are defined in the usual way. 
For example, $\mathbf{Comon}(\ca{V})$ has as
objects triples $(C,\Delta,\epsilon)$ where $C\in\mathrm{ob}\ca{V}$, 
$\Delta:C\to C\otimes C$ is the \emph{comultiplication} 
and $\epsilon:C\to I$ is the \emph{counit}, such that the diagrams
\begin{displaymath} 
\xymatrix {C\ar[r]^-{\Delta}\ar[d]_-{\Delta} &
C\otimes C\ar[d]^-{1\otimes\Delta}\\ C\otimes
C\ar[r]^-{\Delta\otimes1} & C\otimes C\otimes C}\qquad
\mathrm{and}\qquad 
\xymatrix {I\otimes C\ar[dr]_-{l_C} & C\otimes
C\ar[l]_-{\epsilon\otimes1}\ar[r]^-{1\otimes\epsilon} & C\otimes I
\ar[dl]^-{r_C}\\ & C\ar[u]_-{\Delta} &}
\end{displaymath} 
commute, and has as morphisms
$(C,\Delta,\epsilon)\to (C',\Delta',\epsilon')$ arrows $f:C\to C'$ in 
$\ca{V}$ such that
\begin{displaymath} 
\xymatrix {C\ar[r]^-{\Delta}\ar[d]_-{f} & C\otimes
C\ar[d]^-{f\otimes f}\\ C'\ar[r]^-{\Delta'} & C'\otimes
C'}\qquad\mathrm{and}\qquad \xymatrix{ C\ar[r]^-{\epsilon}\ar[d]_-{f}
& I\\ C'\ar[ur]_-{\epsilon'} &}
\end{displaymath} 
commute. The category of monoids $\mathbf{Mon(\ca{V})}$
in a monoidal category is defined dually, with objects $(A,m,\eta)$
where $A\in\mathrm{ob}\ca{V}$, $m:A\otimes A\to A$ is the \emph{multiplication}
and $\eta:I\to A$ is the \emph{unit}.

Both the categories of monoids and comonoids of a symmetric monoidal
category $\ca{V}$ are themselves symmetric monoidal categories, 
the tensor product and the symmetry
inherited from $\ca{V}$. For example, $\mathbf{Mon(\ca{V})}$ 
is monoidal via the monoid structure
\begin{displaymath}
A\otimes B\otimes A\otimes B\xrightarrow{1\otimes s\otimes 1} 
A\otimes A\otimes B\otimes B\xrightarrow{m_A\otimes m_B}A\otimes B
\end{displaymath}
on $A\otimes B$ for $A,B\in\mathbf{Mon(\ca{V})}$, where $s$ is the 
symmetry in $\ca{V}$. 

There are obvious forgetful functors $\mathbf{Mon(\ca{V})}\to\ca{V}$, 
$\mathbf{Comon(\ca{V})}\to\ca{V}$, and these often have a left/right 
adjoint respectively, sending an object of $\ca{V}$ to the free 
monoid/cofree comonoid on that object. There exist various conditions on $\ca{V}$
that guarantee the existence of free monoids. 
For example, if $\ca{V}$ has countable coproducts 
which are preserved by the tensor product (on either side), then the free monoid
on $X\in \mathrm{ob}\ca{V}$ is given by $\coprod_{n\in\mathbb{N}} X^{\otimes n}$.

Recall that given two (ordinary) adjoint functors 
$F\dashv G$ between monoidal categories, colax monoidal structures on 
the left adjoint $F$ correspond bijectively, via mates 
under the adjunction, to lax monoidal structures on 
the right adjoint $G$. This is a well-known result, coming from
the so-called \emph{doctrinal adjunction} \cite{Doctrinal}. 

For example, in a symmetric monoidal closed category $\ca{V}$, since
$-\otimes A\dashv[A,-]$ for all $A$, the bifunctor 
$[-,-]:\ca{V}^\mathrm{op}\times\ca{V}\to\ca{V}$ is the
\emph{parametrized adjoint} (see~\cite{MacLane})
of the tensor product functor
$\otimes:\ca{V}\times\ca{V}\to\ca{V}$. Since the tensor
product is a (strong) monoidal functor, hence lax and colax, via
\begin{displaymath} 
\phi_{(A,B),(A',B')}:A\otimes B\otimes
A'\otimes B'\stackrel{\sim}{\longrightarrow}A\otimes A'\otimes
B\otimes B' 
\quad\mathrm{and}\quad
 \phi_0:I\stackrel{\sim}{\longrightarrow}I\otimes I,
\end{displaymath} we get:
\begin{cor}
In a symmetric monoidal closed category $\ca{V}$, the 
internal hom functor 
$[-,-]:\ca{V}^\mathrm{op}\otimes\ca{V}\to\ca{V}$
is a lax monoidal functor, with lax structure maps
$\psi_{(A,B),(A',B')}:[A,B]\otimes[A',B']\to[A\otimes A',B\otimes B']$,
$\psi_0:I\to[I,I]$
corresponding under the adjunction $-\otimes X\dashv[X,-]$ to 
$[A,B]\otimes[A',B']\otimes A\otimes A'\xrightarrow{1\otimes s\otimes 1}
[A,B]\otimes A\otimes[A',B']\otimes A'
\xrightarrow{e\otimes e}B\otimes B'$
and $I\otimes I\xrightarrow{\sim} I$.
\end{cor}

An important property of lax monoidal functors is that they map
monoids to monoids. 
In particular, the internal hom 
for a symmetric
monoidal closed category $\ca{V}$ 
takes monoids to monoids, 
and since we have
$\mathbf{Mon}(\ca{V}^{\mathrm{op}}\times\ca{V})
\cong\mathbf{Comon}(\ca{V})^\mathrm{op}\times\mathbf{Mon}(\ca{V})$,
we obtain a functor
\begin{equation}\label{eq:defH} 
\xymatrix @R=.05in
{\mathbf{Mon}([-,-])=H:\mathbf{Comon}(\ca{V})^{\mathrm{op}}
\times\mathbf{Mon}(\ca{V})\ar[r]
&\mathbf{Mon}(\ca{V})\\ \qquad\qquad\qquad\qquad\quad
\qquad\qquad\qquad\qquad(C,A)\ar @{|->}[r] & [C,A].\quad}
\end{equation}

The concrete content of this observation is that whenever $C$ 
is a comonoid and $A$ a monoid, the object $[C,A]$ is endowed with
the structure of a monoid, with unit $I\to [C,A]$ which is the 
transpose, under $-\otimes C\dashv [C,-]$, of 
$C\xrightarrow{\epsilon}I\xrightarrow{\eta}A,$
and with multiplication $[C,A]\otimes[C,A]\to[C,A]$ the 
transpose of 
\begin{displaymath}
\xymatrix @C=.4in @R=1pc {[C,A]\otimes[C,A]\otimes C
\ar[r]^-{1\otimes 1\otimes \Delta}
\ar @{-->}[ddrr]
& [C,A]\otimes[C,A]\otimes C\otimes C\ar[r]^-{1\otimes s\otimes 1}
& [C,A]\otimes C\otimes [C,A]\otimes C\ar[d]
^-{e\otimes e}\\
&& A\otimes A\ar[d]^-{m}\\
&& A.}
\end{displaymath}

\subsection{Categories of Modules and Comodules}\label{catsmodscomods}

The categories of modules of a monoid and comodules of a 
comonoid in a monoidal category $\ca{V}$ are defined in the usual 
way. More precisely, 
for each monoid $(A,m,\eta)\in\mathbf{Mon}(\ca{V})$ there is a category
$\Mod_{\ca{V}}(A)$ of (left) $A$-modules, with objects $(M,\mu)$, 
where $M$ is an object in $\ca{V}$ and $\mu:A\otimes M\to M$ is the 
\emph{action}, an arrow such that the diagrams
\begin{equation}\label{eq:defmod} 
\xymatrix {A\otimes A\otimes M\ar[r]^-{m\otimes
1}\ar[d]_-{1\otimes\mu} & A\otimes M\ar[d]^-{\mu}\\ A\otimes
M\ar[r]^-{\mu} & M}\qquad\mathrm{and}\qquad\xymatrix {& A\otimes
M\ar[rd]^-{\mu} &\\ I\otimes M\ar[rr]^-{l_M}\ar[ur]^-{\eta\otimes1} && M}
\end{equation} 
commute, and morphisms $(M,\mu_M)\to(N,\mu_N)$ are arrows $u:M\to N$ in $\ca{V}$ such that
\begin{equation}\label{eq:defmod2}
 \xymatrix {A\otimes M\ar[r]^-{1\otimes
u}\ar[d]_-{\mu_M} & A\otimes N\ar[d]^-{\mu_N}\\ M\ar[r]^-u & N}
\end{equation} 
commutes. The category of (right) comodules
$\Comod_{\ca{V}}(C)$ for a comonoid $(C,\Delta,\varepsilon)$
in $\mathbf{Comon}(\ca{V})$ is defined dually, with objects
$(X,\delta)$, where $X$ is an object in $\ca{V}$ and 
$\delta:X\to X\otimes C$ is the \emph{coaction}, sastifying dual axioms.
It is well-known that the forgetful functors 
$\mathbf{Mod_{\ca{V}}}(A)\xrightarrow{V_A}\ca{V}$ 
and $\mathbf{Comod_{\ca{V}}}(C)\xrightarrow{U_C}\ca{V}$ which simply
forget the module action/comodule coaction, are respectively 
monadic/comonadic.

Each monoid arrow $f:A\to B$ determines a functor  
\begin{equation}\label{eq:res}
\Mod(f)=f^{\sharp}:\Mod_{\ca{V}}(B)\to\Mod_{\ca{V}}(A)
\end{equation} 
sometimes called \emph{restriction of scalars},
which makes every $B$-module $N$ into an $A$-module
$f^{\sharp}N$ via the action 
$A\otimes N\xrightarrow{f\otimes1}
B\otimes N\xrightarrow{\mu}N.$
Also, each $B$-module arrow becomes
an $A$-module arrow, and so we have a commutative triangle of categories
and functors
\begin{equation}\label{eq:tr1}
\xymatrix @R=.25in{\Mod_{\ca{V}}(B)\ar[rr]^-{f^{\sharp}}\ar[dr]_-{V_B} &&
\Mod_{\ca{V}}(A)\ar[dl]^-{V_A}\\ & \ca{V}. &}
\end{equation}
Dually, each comonoid arrow $g:C\to D$
induces a functor
\begin{equation}\label{eq:cores}
\Comod(g)=g_*:\Comod_{\ca{V}}(C)\to\Comod_{\ca{V}}(D)
\end{equation} 
called \emph{corestriction of scalars}, which makes every $C$-comodule $X$ into a 
$D$-comodule $g_*X$ via the coaction
$X\xrightarrow{\delta}X\otimes C\xrightarrow{1\otimes g}X\otimes D.$
The respective commutative triangle is
\begin{equation}\label{eq:tr2}
\xymatrix @R=.25in{\Comod_{\ca{V}}(C)\ar[rr]^-{g_*}\ar[dr]_-{U_C} &&
\Comod_{\ca{V}}(D)\ar[dl]^-{U_D}\\ & \ca{V}. &}
\end{equation}
By the above commutative triangles, where the legs are monadic 
and comonadic respectively, $f^\sharp$ is always continuous and 
$g_*$ always cocontinuous.

We saw in the previous section how any lax monoidal functor
$F:\ca{V}\to\ca{W}$ induces a functor
$\mathbf{Mon}(F):\mathbf{Mon}(\ca{V})\to\mathbf{Mon}(\ca{W})$. 
Furthemore, this carries over to the categories of modules. 
If $M\in\Mod_{\ca{V}}(A)$ for a
monoid $A$, then $FM\in\Mod_{\ca{W}}(FA)$, via the
action
\begin{displaymath} 
FA\otimes FM\xrightarrow{\phi_{A,M}}
F(A\otimes M)\xrightarrow{F\mu}FM.
\end{displaymath} 
Again, we can apply this to the lax monoidal internal hom in a symmetric 
monoidal closed category $\ca{V}$, 
hence there is an induced functor for
$C\in\mathbf{Comon}(\ca{V})$ and
$A\in\mathbf{Mon}(\ca{V})$,
\begin{equation} \label{eq:hommodg}
\xymatrix @R=.05in
{\Comod_{\ca{V}}(C)^\mathrm{op}\times\Mod_{\ca{V}}(A)\ar[r]
& \Mod_{\ca{V}}([C,A])\\ \qquad\qquad\qquad\qquad\quad(X,M)\ar
@{|->}[r] & [X,M].\quad\qquad}
\end{equation}
Concretely, this means that whenever $(X,\delta)$ is a $C$-comodule
and $(M,\mu)$ is an $A$-module, the object $[X,M]$ obtains the structure 
of a $[C,A]$-module, with action $[C,A]\otimes[X,M]\to [X,M]$ 
which is the transpose, under $-\otimes X\dashv [X,-]$, of 
\begin{equation}\label{eq:lala} 
\xymatrix @R=1pc {[C,A]\otimes
[X,M]\otimes X\ar[r]^-{1\otimes1\otimes\delta}\ar @{-->}[ddrr] &
[C,A]\otimes [X,M]\otimes X\otimes
C\ar[r]^-{1\otimes s\otimes 1} &
[C,A]\otimes C\otimes
[X,M]\otimes X\ar[d]^-{e_A\otimes e_M}\\ 
&& A\otimes
M\ar[d]^-{\mu}\\ 
&& M.}
\end{equation}

\subsection{Admissible Categories}\label{admcats} 
We mentioned earlier that the existence 
of the free monoid/cofree comonoid functor depends on specific 
assumptions on the category $\ca{V}$. We now consider a class of monoidal categories
which provides good results regarding the formation of limits and 
colimits of the categories of monoids and comonoids, as well as properties 
which permit the existence of various adjunctions.

Recall (see~\cite{MR1294136}) that 
a \emph{locally $\lambda$-presentable} 
category $\ca{K}$ is a cocomplete category
which has a set $\ca{A}$ of $\lambda$-presentable objects such that
every object is a $\lambda$-filtered colimit of objects from $\ca{A}$. 
In such a category $\ca{K}$, all
$\lambda$-presentable objects have a set of representatives
of isomorphism classes of objects. Any such set is denoted by
$\mathbf{Pres}_{\lambda}\ca{K}$ and is a small dense full subcategory
of $\ca{K}$, hence also a strong generator. 
Other useful properties of locally presentable categories are
completeness, wellpoweredness and co-wellpoweredness.  

Following the approach of \cite{MR2529129}, consider the class of
\emph{admissible} monoidal categories, i.e. locally presentable
symmetric monoidal categories $\ca{V}$, such that for each object $C$,
the functor $C\otimes -$ preserves filtered colimits. Examples of this
class of categories is the category $\Mod_R$ for a commutative ring $R$,
every locally presentable category with respect to cartesian products 
and every symmetric monoidal closed category that is locally presentable.

Then, for an admissible monoidal category
$\ca{V}$, it is shown in \cite{MR2529129}  
that $\mathbf{Mon}(\ca{V})$ is
monadic over $\ca{V}$ and locally presentable itself and also
$\mathbf{Comon}(\ca{V})$ is a locally presentable category 
and comonadic over $\ca{V}$.

The following simple adjoint functor theorem, identified by Max Kelly,
will be used repeatedly.
\begin{thm}~\cite[5.33]{Kelly}\label{Kellythm}
If the cocomplete $\ca{C}$ has a small dense subcategory, every
cocontinuous $S:\ca{C}\to\ca{B}$ has a right adjoint.
\end{thm}
As an application of this theorem, consider the category $\mathbf{Comon}(\ca{V})$ 
for a locally presentable
symmetric monoidal closed $\ca{V}$. Then:
\begin{itemize}
\item[-] $\mathbf{Comon}(\ca{V})$ is cocomplete.
\item[-] $\mathbf{Comon}(\ca{V})$ has a small dense subcategory,
$\mathbf{Pres}_{\lambda}\mathbf{Comon}(\ca{V})$,
since it is locally presentable.
\item[-] The functor $-\otimes
C:\mathbf{Comon}(\ca{V})\to\mathbf{Comon}(\ca{V})$ is cocontinuous,
via the commutative diagram
$\xymatrix{\mathbf{Comon}(\ca{V})\ar[r]^-{-\otimes
C}\ar[d]_-F & \mathbf{Comon}(\ca{V})\ar[d]^-F\\
\ca{V}\ar[r]^-{-\otimes FC} & \ca{V}}$,
where $(-\otimes FC)\circ F$ is cocontinuous and the
forgetful $F$ creates colimits. 
\end{itemize}
Therefore the functor 
$-\otimes C$ has a right adjoint, and hence
the following result holds (see also \cite[3.2]{MR2529129}):
\begin{prop}\label{coalgmonclosed}
If $\ca{V}$ is a 
locally presentable symmetric monoidal closed 
category, the category of comonoids 
$\mathbf{Comon}(\ca{V})$ 
is a monoidal closed category as well.
\end{prop}

\subsection{The category $\Mod_R$} \label{Modcat}
Consider the category $\Mod_R$ of $R$-modules and $R$-module maps, 
for $R$ a commutative ring. It is of course a symmetric monoidal category, 
with the usual tensor product of $R$-modules.
It is also monoidal closed, by the well-known adjunction
$\xymatrix @C=3.5pc{\Mod_R\ar @<+.7ex>[r]^-{-\otimes N}
_-*-<5pt>{\text{\scriptsize{$\bot$}}} 
& \Mod_R,\ar @<+.7ex>[l]^-{\Hom_R(N,-)}}$ 
and moreover a locally presentable category, complete and cocomplete.

The categories of monoids and comonoids in $\Mod_R$
are $\mathbf{Mon}(\Mod_R)=\Alg_R$ and
$\mathbf{Comon}(\Mod_R)=\Coalg_R$. 
Based on (\ref{eq:defH}),
the internal hom of the category induces the functor
\begin{equation} 
\xymatrix
@R=.05in{\Hom_R:\Coalg_R^{\mathrm{op}}\times\Alg_R\ar[r]
&\Alg_R\qquad\quad\\ 
\quad\qquad\qquad\qquad\quad(C,A)\ar @{|->}[r] & \Hom_R(C,A).}
\end{equation}
This is the same as the
well-known fact that for $C$ an $R$-coalgebra and $A$ an
$R$-algebra, $\Hom_R(C,A)$ obtains
the structure of an $R$-algebra under
the \emph{convolution product}
\begin{equation}\label{eq:conv} 
(f*g)(c)=\sum f(c_1)g(c_2)\quad\mathrm{and}\quad
1=\eta\circ\epsilon.
\end{equation}

Now, since $\Mod_R$ is an admissible category in the sense of 
the previous section, we deduce that 
$\Alg_R$ is a locally presentable category and monadic 
over $\Mod_R$, and $\Coalg_R$ is 
comonadic over $\Mod_R$, locally
presentable and monoidal closed.

Denote by $\Comod_C$ the category of $R$-modules which 
have a $C$-comodule structure
for an $R$-coalgebra $C$, and respectively $\Mod_A$ 
the category of $A$-modules. We know that these categories possess
many useful properties (see e.g.~\cite{MR0414567}). In particular,
apart from the facts that $\Mod_A$ is monadic and $\Comod_C$ is comonadic 
over $\Mod_R$,
we also have that $\Comod_C$ is complete, wellpowered and co-wellpowered, 
has a generator and a cogenerator and is locally presentable.

Regarding the restriction and corestriction of scalars in this case, 
for $\ca{V}=\Mod_R$,
$f:A\to B$ in $\Alg_R$ and
$g:C\to D$ in $\Coalg_R$, 
the triangles (\ref{eq:tr1}), (\ref{eq:tr2}) become 
$\qquad\xymatrix @R=.22in @C=.25in{\Mod_B\ar[rr]^-{f^{\sharp}}\ar[dr]_-{V_B} &&
\Mod_A\ar[dl]^-{V_A}\\ & \Mod_R &} \,\ $ and 
$ \,\ \xymatrix @R=.22in @C=.20in {\Comod_C\ar[rr]^-{g_*}\ar[dr]_-{U_C} &&
\Comod_D\ar[dl]^-{U_D}\\ & \Mod_R &}$.\\
Since $\Mod_R$ is a symmetric monoidal closed category,
with all limits and colimits, we obtain
a pair of adjoints 
$\xymatrix
{\Mod_B \ar[rr]|-{f^\sharp} &&
\Mod_A \ar @/_4ex/[ll]_-{f_\sharp}
^-{\bot}
\ar @/^4ex/[ll]^-{\tilde{f}}
_-{\bot}}$
for the restriction of scalars, 
the left adjoint
$f_\sharp\dashv f^\sharp$ given by 
$f_\sharp\cong B\otimes_A -:\Mod_A
\to\Mod_B$
where $B$ is regarded as a left-$B$ right-$A$ bimodule, 
and the right adjoint $f^\sharp\dashv \tilde{f}$
given by 
$\tilde{f}\cong \Hom_A(B,-):\Mod_A\to
\Mod_B$
where $B$ is regarded as a left-$A$ right-$B$ bimodule.

As for the corestriction of scalars, because of the 
properties of $\Comod_C$ mentioned above, there 
exists a right adjoint $g_*\dashv g^*$, which is given, when $C$ is a 
flat $R$-coalgebra, by  $g^*\cong -\square_D C:\Comod_D\to
\Comod_C,$
where $C$ is regarded as a left-$D$ right-$C$ bicomodule and
$\square$ is the cotensor product (e.g. see ~\cite{MR0414567}).
   
Next, we can apply (\ref{eq:hommodg}) to $\Mod_R$. For 
$C$ an $R$-coalgebra and $A$ an $R$-algebra, the induced map, 
denoted by $\Hom$, is
\begin{equation}\label{eq:hommod} 
\xymatrix @R=.05in
{\Hom:\Comod_C^\mathrm{op}\times \Mod_A \ar[r] &
\Mod_{\Hom_R(C,A)}\qquad\\
\qquad\quad\qquad\qquad\quad(X,M)\ar @{|->}[r] & \Hom_R(X,M) \,\ \qquad}
\end{equation}
This means that whenever $(X,\delta)$ is a $C$-comodule
and $(M,\mu)$ is an $A$-module, then $(\Hom_R(X,M),\mu')$ is a $\Hom_R(C,A)$-module,
denoted just by $\Hom(X,M)$, the action $\mu'$ described as in (\ref{eq:lala}).

Moreover, the functor
$\Hom(-,M):\Comod_C^\mathrm{op}\to \Mod_{\Hom_R(C,A)}$
for any $M\in\Mod_A$ is continuous, which is obvious by the following 
commutative square, where $\Hom_R(-,UM)\circ U^\mathrm{op}$ preserves limits and the 
forgetful $V$ creates them:
\begin{equation}\label{eq:sth3} 
\xymatrix @C=.7in
@R=.3in{\Comod_C^\mathrm{op}\ar[r]^-{\Hom(-,M)}\ar[d]_-{U^\mathrm{op}} &
\Mod_{\Hom_R(C,A)}\ar[d]^-{V}\\
\Mod_R^\mathrm{op}\ar[r]^-{\Hom_R(-,UM)} &\Mod_R.}
\end{equation}
Similarly, for each $D$-comodule $Y$, the evident functor
$-\otimes
Y:\Comod_C\longrightarrow \Comod_{C\otimes D}$
 is cocontinuous, as it is clear from the
following commutative diagram:  
\begin{equation}\label{eq:sth}
\xymatrix @C=.7in @R=.3in
{\Comod_C\ar[r]^-{-\otimes Y}\ar[d]_-{U} &
\Comod_{C\otimes D}\ar[d]^-U\\ 
\Mod_R
\ar[r]^-{-\otimes UY} & \Mod_R.}
\end{equation}

\subsection{Actions of a Monoidal Category}\label{actions}
Recall that an \emph{action} of a monoidal category
$\ca{V}={(\ca{V},\otimes,I,a,l,r)}$ on a category $\ca{D}$ is given
by a functor $*:\ca{V}\times\ca{D}\to\ca{D}$ written $(X,D)\mapsto
X*D$, a natural isomorphism with components 
${\alpha}_{XYD}:(X\otimes Y)*D\stackrel{\sim}{\longrightarrow} X*(Y*D)$, 
and a natural isomorphism with components
${\lambda}_D :I*D\stackrel{\sim}{\longrightarrow} D$, 
satisfying the commutativity of the diagrams
\begin{equation}\label{eq:diag1} 
\xymatrix{((X\otimes Y)\otimes
Z)*D\ar[r]^-{\alpha}\ar[d]_-{a*1} & (X\otimes Y)*(Z*D)\ar[r]^-{\alpha}
& X*(Y*(Z*D))\\ (X\otimes(Y\otimes Z))*D\ar[rr]^-{\alpha} &&
X*((Y\otimes Z)*D),\ar[u]^-{1*\alpha}}
\end{equation}
\begin{equation}\label{eq:diag2} 
\xymatrix{(I\otimes
X)*D\ar[rr]^-{\alpha}\ar[dr]_-{l*1} && I*(X*D)\ar[dl]^-{\lambda}\\ &
X*D, &}
\end{equation}
\begin{equation}\label{eq:diag3} 
\xymatrix{(X\otimes
I)*D\ar[rr]^-{\alpha}\ar[dr]_-{r*1} && X*(I*D)\ar[dl]^-{1*{\lambda}}\\
& X*D. &}
\end{equation}

\begin{rmk} A monoidal category $\ca{V}$ is actually a 
\emph{pseudomonoid} inside the monoidal category 
$(\mathbf{Cat},\times,1)$, so this action described above is 
exactly the definition of a \emph{pseudoaction} of a 
pesudomonoid on an object of $\mathbf{Cat}$.
\end{rmk}

The most important fact here, explained in detail in \cite{MR1897810}, 
is that to give a category $\ca{D}$ and an action of a monoidal
closed $\ca{V}$ with a right adjoint for each $-*D$ is to give
a tensored $\ca{V}$-category:
 
\begin{prop}\label{actionenrich}
Suppose that $\ca{V}$ is a monoidal category which acts on 
a category $\ca{D}$ via 
the bifunctor $H:\ca{V}\times\ca{D}\to\ca{D}$, and $H(-,B)$ has
a right adjoint $H(-,B)\dashv F(B,-)$ for every $B\in\ca{D}$, with the
natural isomorphism
\begin{equation}\label{eq:adjact}
\ca{D}(H(X,B),D)\cong\ca{V}(X,F(B,D)).
\end{equation}
Then, we can enrich $\ca{D}$ in $\ca{V}$, in the sense that 
there is a $\ca{V}$-category $\ca{K}$,
with the same objects as $\ca{D}$, hom-objects
$\ca{K}(A,B)=F(A,B)$ and underlying category
$\ca{K}_o=\ca{D}$.
\end{prop}

The proof that there exists a composition functor 
$M:F(B,C)\otimes F(A,B)\to F(A,C)$ and identity elements
$j_A:I\to F(A,A)$ satisfying the usual axioms of 
enriched categories relies on just the correspondence of 
arrows under the adjunction (\ref{eq:adjact}) and the
action properties. We denote the $\ca{V}$-category $\ca{K}$
again by $\ca{D}$.

When $\ca{V}$ is monoidal closed,
we get a natural isomorphism 
$\ca{D}(H(X,B),D)\cong[X,F(B,D)]$,  
and so
this $\ca{V}$-enriched representation defines the 
\emph{tensor
product} (see \cite{Kelly}) in $\ca{D}$ of $X$ and $B$, 
namely $H(X,B)$.

As a special case of this result, consider the situation when 
$\ca{V}=\ca{C}$, $\ca{D}=\ca{A}^\mathrm{op}$, and the action 
is the bifunctor $H:\ca{C}\times\ca{A}^\mathrm{op}\to\ca{A}^\mathrm{op}$. 
Note that, via $H(-,B)\dashv F(B,-)$,
the functor $F$ is actually the parametrized adjoint of $H$, so
it is also a bifunctor $F:\ca{A}\times\ca{A}^\mathrm{op}\to\ca{C}$. 
Denote this $F$ as 
$P:\ca{A}^\mathrm{op}\times\ca{A}\to\ca{C}$ and 
so the natural isomorphism (\ref{eq:adjact}) of the adjunction 
$H(-,B)\dashv P(-,B)$ becomes
\begin{displaymath}
\ca{A}^\mathrm{op}(H(C,B),A)\cong\ca{C}(C,P(A,B)).
\end{displaymath}

\begin{cor}\label{importcor1}
When we have an action $H:\ca{C}\times\ca{A}^\mathrm{op}
\to\ca{A}^\mathrm{op}$ 
of the monoidal $\ca{C}$ on the category $\ca{A}^\mathrm{op}$
along with an adjunction $\xymatrix{\ca{C}\ar@<+.7ex>[r]^-{H(-,B)} 
_-*-<5pt>{\text{\scriptsize{$\bot$}}}&
\ca{A}^\mathrm{op}\ar@<+.7ex>[l]^-{P(-,B)}}$ for each $B$, then $\ca{A}^\mathrm{op}$  
is enriched in $\ca{C}$ with hom-objects $\ca{A}^\mathrm{op}(A,B)=P(B,A)$, 
and $H(C,B)$ is the tensor product of $C$ and $B$. 
\end{cor} 

Moreover, when the monoidal 
category $\ca{C}$ is symmetric, then $\ca{A}=(\ca{A}^\mathrm{op})^\mathrm{op}$
is also enriched in $\ca{C}$, with the same objects and 
hom-objects 
$\ca{A}(A,B)=\ca{A}^\mathrm{op}(B,A)$.
Hence,

\begin{cor}\label{importcor2}
If $\ca{C}$ is symmetric and the above conditions hold, 
then $\ca{A}$ is enriched in $\ca{C}$, with 
hom-objects $\ca{A}(A,B)=P(A,B)$.
\end{cor}

\section{The existence of the universal measuring coalgebra}\label{existmeascoal}

As we mentioned in the introduction, we explore the 
existence of an object $P(A,B)$ 
and a natural isomorphism 
\begin{equation}\label{this}
\Alg_k(A,\Hom_k(C,B))\cong\Coalg_k(C,P(A,B))
\end{equation}
defining the universal measuring coalgebra, 
but for general categories of monoids and comonoids
in some monoidal category $\ca{V}$. Since the functor 
$\Hom_k(-,-)$ is in fact the internal hom, algebras 
are the monoids and coalgebras  are the comonoids in $\mathbf{Vect_k}$,
this evidently comes down to the existence of an adoint of 
the internal hom functor applied on comonoids in the 
first variable.  

Consider a symmetric monoidal closed category $\ca{V}$. 
We saw in the previous section that a functor
$H:\mathbf{Comon}(\ca{V})^{\mathrm{op}}\times
\mathbf{Mon}(\ca{V})\to \mathbf{Mon}(\ca{V})$
is induced by the internal hom (see (\ref{eq:defH})). 
In order to find a right adjoint for
\begin{displaymath}
H(-,B)^{\mathrm{op}}:\mathbf{Comon}(\ca{V})
\to\mathbf{Mon}(\ca{V})^{\mathrm{op}}
\end{displaymath} we can use Theorem \ref{Kellythm}. For that we need
$H(-,B)^{\mathrm{op}}$ to be a cocontinuous functor and 
$\mathbf{Comon}(\ca{V})$ to be a cocomplete category with
a small dense subcategory.

For example, in the case of
$\ca{V}=\mathbf{Vect_k}$, one can
easily see that the above conditions hold: $\mathbf{Coalg_k}$ is a
cocomplete category and has a small dense subcategory, namely the
coalgebras which are finitely dimensional as vector spaces over $k$, 
and the hom functor is continuous. So we recover (\ref{this}).

We saw in Section \ref{admcats} that if we consider the 
class of admissible monoidal categories, we 
already know that $\mathbf{Comon}(\ca{V})$ is comonadic over $\ca{V}$,
hence cocomplete, and is also a locally presentable category, 
hence it has a small dense subcategory. Moreover, the diagram
\begin{equation}\label{eq:Hcont} 
\xymatrix{\mathbf
{Comon}(\ca{V})^{\mathrm{op}}\ar[rr]^-{H(-,B)}\ar[d]_-{F_1^\mathrm{op}} &&
\mathbf{Mon}(\ca{V})\ar[d]^-{F_2}\\
\ca{V}^{\mathrm{op}}\ar[rr]^-{[-,F_2B]} &&\ca{V}}
\end{equation} 
commutes, where $F_1,F_2$ are the respective 
forgetful functors. The functor $[-,F_2B]$ is continuous 
as the right adjoint
of $[-,F_2B]^{\mathrm{op}}$, so the composite 
$[-,F_2B]\circ F_1^\mathrm{op}$ preserve all limits, and also
the monadic $F_2$ creates limits. Therefore $H(-,B)$ is 
continuous, and so its opposite functor 
$H(-,B)^{\mathrm{op}}$ is
cocontinuous.

Since all the required conditions of the theorem 
are satisfied, we obtain the following result.

\begin{prop}
For an admissible monoidal category $\ca{V}$, the 
functor $H(-,B)^{\mathrm{op}}$ has a right adjoint, namely
$P(-,B):\mathbf{Mon}(\ca{V})^{\mathrm{op}}
\to\mathbf{Comon}(\ca{V})$,
and a natural isomorphism
\begin{equation}\label{eq:genmeas}
\mathbf{Comon}(\ca{V})(C,P(A,B))\cong
\mathbf{Mon}(\ca{V})^{\mathrm{op}}(H(C,B)^{\mathrm{op}},A)
\end{equation}
is established.
\end{prop}

In particular, since the category $\Mod_R$ is an admissible
monoidal category, we can apply the above to get an adjunction
$\xymatrix
@C=.6in{\Coalg_R\ar@<+.7ex>[r]^-{\Hom_R(-,B)^{\mathrm{op}}} 
 _-*-<5pt>{\text{\scriptsize{$\bot$}}} &
\Alg_R^{\mathrm{op}}\ar@<+.7ex>[l]^-{P(-,B)}}$ given by
\begin{equation}\label{eq:meascoal}
\Coalg_R(C,P(A,B))\cong\Alg_R(A,\Hom_R(C,B)).
\end{equation}
We call the object $P(A,B)$ in $\Coalg_R$ the
\emph{universal measuring coalgebra}.
We obtain a bifunctor 
\begin{equation}\label{eq:exP}
P:\Alg_R^\mathrm{op}\times\Alg_R
\to\Coalg_R
\end{equation}
such that the above isomorphism is natural in all three 
variables.

\begin{rmk} The functor $\Hom_k(-,k)$ for a field $k$
is the so-called `dual algebra functor', taking 
any coalgebra $C$ to its dual $C^*$=$\Hom_k(C,k)$, which has 
a natural structure of an algebra. It is well-known
that, since $A^*$ for an algebra $A$ in general fails to have 
a coalgebra structure, we can define
\begin{displaymath}
A^0=\{g\in A^*|kerg\,\ \textrm{contains a cofinite ideal}\}
\end{displaymath}
which turns out to be a coalgebra, so that 
the functors $( \,\ )^0$ and $( \,\ )^*$ are adjoint 
to one another. The adjunction 
(\ref{eq:meascoal}) then, for $B=k$, produces the well-known isomorphism 
\begin{displaymath}
\Coalg_k(C,A^0)\cong\Alg_k(A,C^*)
\end{displaymath} 
therefore $P(A,k)\cong A^0$. So actually (\ref{eq:meascoal}) 
generalizes the dual algebra functor adjunction for $R$
a commutative ring.
\end{rmk}
 
We now proceed to the statement and proof of a lemma that connects the
above adjunction (\ref{eq:meascoal}) 
with the usual $-\otimes_R C\dashv\Hom_R(C,-)$ for arbitrary $R$-modules.

\begin{lem}\label{lemma} 
Suppose we have an $R$-algebra map $f:A\to
\Hom_R(C,B)$, $A,B\in\Alg_R$, $C\in\Coalg_R$. If
it corresponds to $\bar{f}:A\otimes C\to B$ under $-\otimes C\dashv
\Hom_R(C,-)$ and to $\hat{f}:C\to P(A,B)$ under
$\Hom_R(-,B)^{\mathrm{op}}\dashv P(-,B)$, then
$\bar{f}=(\alpha\otimes\hat{f})\circ e$, where $e$ is the
evaluation and $\alpha$ the unit of the second adjunction.
\end{lem}
 
\begin{proof}
\begin{displaymath} 
\xymatrix @C=.25in {& \Hom_R(P(A,B),B)\otimes
C\ar[r]^-{1\otimes\hat{f}}\ar[ddr]^-{\Hom_R(\hat{f},1)\otimes 1} &
\Hom_R(P(A,B),B)\otimes P(A,B)\ar[dr]^-{e_B} &\\ A\otimes
C\ar[ur]^-{\alpha\otimes 1}\ar[drr]_-{f\otimes 1}\ar @{-->}
@/_3ex/[rrr]_{\bar{f}} && & B\\ && \Hom_R(C,B)\otimes
C\ar[ur]_-{e_B} &}
\end{displaymath}
By inspection of the diagram, we can see that the
left part is the commutative diagram which gives $f$ through its
transpose map $\hat{f}$ under $\Hom_R(-,B)^{\mathrm{op}}\dashv P(-,B)$,
and the right part commutes by dinaturality of the counit
$e_D^E:\Hom_R(D,E)\otimes D\to E$ of the parametrized adjunction
$-\otimes -\dashv \Hom_R(-,-)$.
\end{proof}

\section{Enrichment of algebras in coalgebras}\label{enrichmalgscoalgs}

Now that we have established the existence of the measuring coalgebra
$P(A,B)$ in $\Mod_R$ through the isomorphism
$\Coalg_R(C,P(A,B))\cong\Alg_R(A,\Hom_R(C,B))$,
we can combine this result with the theory 
of actions of monoidal categories 
and show that there exists a way to 
enrich the category $\Alg_R$ in the symmetric 
monoidal closed category $\Coalg_R$.

\begin{rmk}\label{actionHom}
For any symmetric monoidal category $\ca{V}$, the internal hom bifunctor
$[-,-]:\ca{V}^\mathrm{op}\times\ca{V}\to\ca{V}$ is an action of 
$\ca{V}^\mathrm{op}$ on $\ca{V}$, 
via the natural isomorphisms 
\begin{displaymath} 
\alpha_{XYZ}:[X\otimes Y,D]
\stackrel{\sim}{\longrightarrow}[X,[Y,Z]]\quad\mathrm{and}
\quad\lambda_D:[I,D]\stackrel{\sim}{\longrightarrow}D.
\end{displaymath}
Hence the functor
$\Hom_R:\Mod_R^\mathrm{op}
\times\Mod_R\to\Mod_R$ is an action (and 
$\Hom_R^\mathrm{op}$ is an action too).
\end{rmk}

If we take $\ca{A}=\Alg_R$ and
$\ca{C}=\Coalg_R$ in the statement of Corollary \ref{importcor1}, 
we do have a bifunctor
\begin{displaymath} 
H=\Hom_R^\mathrm{op}:\Coalg_R\times
\Alg_R^\mathrm{op}\to
\Alg_R^\mathrm{op},
\end{displaymath} 
and $\Hom_R(-,B)^\mathrm{op}$ has a right
adjoint $P(-,B)$ (by (\ref{eq:meascoal})). What remains to be 
shown is that the bifunctor
$H=\Hom_R(-,-)$ is an \emph{action} of
$\Coalg_R^\mathrm{op}$ on $\Alg_R$ (so then 
also $\Hom_R^\mathrm{op}$ will be an action of $\mathbf{Coalg_R}$ 
on $\mathbf{Alg_R}^\mathrm{op}$). We are looking for two natural 
isomorphisms in $\Alg_R$
\begin{displaymath}
\alpha_{CDA}:\Hom_R(C\otimes D,A)\xrightarrow{\sim}
\Hom_R(C,\Hom_R(D,A)),\,\ 
\lambda_A:\Hom_R(R,A)\stackrel{\sim}{\to} A
\end{displaymath}
such that the diagrams (\ref{eq:diag1}), (\ref{eq:diag2}), 
(\ref{eq:diag3}) commute. But we already know that these 
isomorphisms exist in $\Mod_R$ and make the functor 
$\Hom_R$ into an action, by the Remark \ref{actionHom} above. 
So it is enough to see that 
these $R$-module isomorphisms in fact lift to $R$-algebra isomorphisms
(since $\Alg_R$ is monadic over $\Mod_R$),
where $\Hom_R(C\otimes D,A)$, $\Hom_R(C,\Hom_R(D,A))$ and $\Hom_R(R,A)$
are algebras via convolution (see (\ref{eq:conv})).

Furthemore, the
diagrams which define an action commute, since they 
do for all $R$-modules. Therefore the bifunctor
$\Hom_R:\Coalg_R^\mathrm{op}\times
\Alg_R\to\Alg_R$ is indeed an action,
and so Corollaries \ref{importcor1}
and \ref{importcor2} apply on this case: 
\begin{prop}
The category $\Alg_R^\mathrm{op}$ is enriched in 
$\Coalg_R$, with hom-objects 
$\Alg_R^\mathrm{op}(A,B)=P(B,A)$,
and $\Hom_R(C,B)$ is the tensor product of $C$ and $B$. 
\end{prop}

\begin{prop}
The category $\Alg_R$
is enriched in $\Coalg_R$, with hom-objects
$\Alg_R(A,B)=P(A,B)$.
\end{prop}

\section{The categories $\Comod$ and $\Mod$}\label{catsModComod}

We consider the \emph{global category of comodules} and the \emph{global 
category of modules} over a commutative
ring $R$. In reality, both definitions are derived from the 
well-known fact that, via the Grothendieck 
construction, there is a correspondance between pseudofunctors 
$\ca{H}:\ca{B}^\mathrm{op}\to\mathbf{Cat}$
and cloven fibrations 
$P=\int{\ca{H}}:\ca{X}\to\ca{B}$,
and also between functors $\ca{H}$ and split fibrations 
$\int{\ca{H}}$ (and dually for opfibrations). Hence, the two functors
\begin{displaymath}
\ca{H}:\Alg_R^\mathrm{op}\to\mathbf{Cat}
\quad\textrm{and}\quad
\ca{G}:\Coalg_R\to\mathbf{Cat}
\end{displaymath} 
which send an algebra $A$ to $\Mod_A$ and a coalgebra $C$ to $\Comod_C$,
based on the well-behaved restriction and corestriction of scalars 
(see (\ref{eq:res}) and (\ref{eq:cores})), give rise to the global categories and 
the functors $V:\Mod\to\Alg_R$, $U:\Comod\to\Coalg_R$.
\\
The explicit definition of
$\Comod$ is the following:
\begin{itemize}
\item Objects: all comodules $X$ over all $R$-coalgebras $C$, denoted by $X_C$.
\item Morphisms: if $X$ is a $C$-comodule and $Y$ is a $D$-comodule, a
map between them is a pair $(k,g):X_C\to Y_D$ where
$\begin{cases}
g_*X\xrightarrow{k}Y & \text{in }\Comod_D\\
\quad C\xrightarrow{g}D & \text{in }\Coalg_R
\end{cases}$.
\item Composition: When we have two morphisms
$\xymatrix{X_C\ar[r]^-{(k,g)} & Y_D\ar[r]^-{(l,h)} & Z_E}$, so
$\begin{cases}
g_*X\xrightarrow{k}Y & \text{in }\Comod_D\\
\quad C\xrightarrow{g}D & \text{in }\Coalg_R
\end{cases}$ \quad\quad and \quad\quad 
$\begin{cases}
h_*Y\xrightarrow{l}Z & \text{in }\Comod_E\\
\quad D\xrightarrow{h}E & \text{in } \Coalg_R
\end{cases},$\\
then their composite $X_C\xrightarrow{(lk,hg)}Z_E$ is
\,\ \,\ $\begin{cases}
(hg)_*X\xrightarrow{lk}Z & \text{in }\Comod_E\\
\quad\quad\,\ C\xrightarrow{hg}E & \text{in } \Coalg_R
\end{cases}$\\
where $lk$ is given by the commutative triangle
$\xymatrix @C=6ex @R=.2in {h_*g_*X\ar @{-->}[r]\ar[dr]_-{h_*k} & Z\\ 
& h_*Y\ar[u]_-{l}}$.
\item[-] Associativity: It holds due to the associativity
of coalgebra and comodule morphisms.
\item[-] Identity: The map 
$\begin{cases}
(1_C)_*X=X\xrightarrow{1_X}X &\text{in }\Comod_C\\
\qquad\qquad\quad C\xrightarrow{1_C}C &\text{in }\Coalg_R
\end{cases} \,\ $ works as
the identity morphism $X_C\xrightarrow{(1_X,1_C)}X_C$ in this category.
\end{itemize}

There is an evident `forgetful' functor
$U:\Comod\to\Coalg_R$ which maps any comodule $X_C$ to its
coalgebra $C$, and any morphism to the coalgebra map of the
pair. This functor is not faithful, since for each coalgebra
map $UX_C\xrightarrow{g} UY_D$ we can choose more than one map in
$\Comod_D(g_*X,Y)$,
but it is full: any
coalgebra map $g:UX_C\to UY_D$ can be written as $U(\varepsilon,g)$, 
for $\begin{cases}
g_*g^*Y\xrightarrow{\varepsilon}Y & \text{in }\Comod_D\\
\qquad C\xrightarrow{g}D & \text{in }\Coalg_R
\end{cases}$, where
$\varepsilon$ is the
counit of the adjunction $g_*\dashv g^*$ (see Section \ref{Modcat}).

We now explore some of the properties of this
category. First of all, $\Comod$ is a 
symmetric monoidal category:
if $X_C,Y_D\in\Comod$, then $X_C\otimes Y_D\in\Comod$, 
being an $R$-module with a coaction
\begin{displaymath}
X\otimes Y\xrightarrow{\delta_X\otimes\delta_Y}
X\otimes C\otimes Y\otimes D
\xrightarrow{\sim} X\otimes Y\otimes C\otimes D 
\end{displaymath} 
over the coalgebra $C\otimes D$
($\Coalg_R$ is monoidal), and the symmetry is inherited from $\Mod_R$. 
Notice that we also have $U(X_C\otimes
Y_D)=C\otimes D=UX_C\otimes UY_D$, which means that the functor $U$
has the structure of a \emph{strict} monoidal functor.

Furthemore, $\Comod$ is a cocomplete
category. This can be shown to be true either via 
the explicit construction of colimits in 
\textbf{Comod}, or as a corollary to a stronger
result:
\begin{prop}\label{Gcomon}
The functor $G:\Comod\to\Mod_R
\times\Coalg_R$, given by
$X_C\mapsto (X,C)$, is comonadic.
\end{prop}
\begin{proof}
Define the functor $H:\Mod_R\times
\Coalg_R\to\Comod$ by 
$(V,D)\mapsto (V\otimes D)_D$ on objects and 
$\xymatrix @C=1in @R=.2in {(V,D)\ar @{|->} [r] \ar[d]^-{(k,f)} &
(V\otimes D)_D\ar[d]^-{(k\otimes f,f)}\\
(W,E)\ar @{|->}[r] & (W\otimes E)_E}$
on arrows. Notice that $V\otimes D$ is a $D$-comodule 
via $V\otimes D\xrightarrow{1\otimes \Delta}V\otimes D\otimes D$, 
and also the linear map $k\otimes f$ is 
in fact an $E$-comodule arrow $f_*(V\otimes D)\to W\otimes E$, 
so the mappings are well-defined.
Then, this functor $H$ is the right adjoint of $G$, via 
the natural bijection
\begin{displaymath}
(\Mod_R\times\Coalg_R)((X,C),(V,D))
\cong\Mod_R(X,V)\times\Coalg_R(C,D)
\end{displaymath}
\begin{equation}\label{eq:bij}
\cong\Comod(X_C,(V\otimes D)_D)
\end{equation}
where $(X,C)=G(X_C)$ and $(V\otimes D)_D=H(V,D)$,
as shown below:\\
i) Given a pair of maps 
$\begin{cases}
k:X\to V &\text{in }\Mod_R\\
f:C\to D &\text{in }\Coalg_R
\end{cases}$,
we obtain the arrow $(\bar{k},f):X_C\to (V\otimes D)_D$
in $\Comod$, where $\bar{k}$ as a linear map is
$X\xrightarrow{\delta}X\otimes C\xrightarrow{k\otimes f}
V\otimes D$.\\
ii) Given $X_C\xrightarrow{(l,g)}(V\otimes D)_D$ in $\Comod$, 
i.e.
$\begin{cases}
g_*X\xrightarrow{l} V\otimes D & \text{in }\Comod_D\\
\quad C\xrightarrow{g} D & \text{in }\Coalg_R
\end{cases}$, we get the pair of arrows
$(X\xrightarrow{l}V\otimes D
\xrightarrow{1\otimes\epsilon}V,C\xrightarrow{g}D)$
in $\Mod_R(X,V)\times\Coalg_R(C,D)$.\\
Moreover, these two directions are inverses to each other,
therefore the bijection (\ref{eq:bij}) is established and we 
obtain the adjunction 
$\xymatrix{\Comod\ar @<+.7ex>[r]^-{G}_-*-<5pt>{\text{\scriptsize{$\bot$}}} 
& \Mod_R\times\Coalg_R\ar @<+.7ex>[l]^-{H}}$.

Hence, a comonad is induced on $\Mod_R\times\Coalg_R$, given by the 
endofunctor $GH$ which maps $(V,D)$ to $(V\otimes D,D)$, the counit
$\varepsilon:GH\Rightarrow id$ with components 
$\varepsilon_{(V,D)}:(V\otimes D,D)\xrightarrow{(1\otimes\epsilon,1)}
(V,D)$ and the natural transformation $G\eta_H:GH\Rightarrow GHGH$ with
components 
$G\eta_{H(V,D)}:(V\otimes D,D)\xrightarrow{(1\otimes\Delta,1)}
(V\otimes D\otimes D,D).$
A coalgebra for this comonad $((V,D),\gamma)$ 
with $(V,D)\xrightarrow{\gamma}(V\otimes D,D)$
turns out to be exactly a $D$-comodule for each different
$(D,\Delta,\epsilon)\in\Coalg_R$, and also a coalgebra arrow 
$((V,D),\gamma)\xrightarrow{(k,f)}((W,E),\gamma')$ turns out to be the same as 
an $E$-comodule morphism $f_*V\to W$.
Hence, the category 
of $GH$-coalgebras $(\Mod_R\times\Coalg_R)^{GH}$ is $\Comod$. 
\end{proof} 

\begin{cor}\label{Comodcocom}
The category $\Comod$ is cocomplete, the functor
$G:\Comod\to\Mod_R
\times\Coalg_R$ creating all colimits.
\end{cor}

For the explicit construction of colimits in $\Comod$, 
consider a diagram $F:J\to\Comod$. 
Since $\Coalg_R$ is cocomplete, the diagram
$J\xrightarrow{F}\Comod\xrightarrow{U}\Coalg_R$
has a colimiting cocone
$(UF_j\xrightarrow{\tau_j} \colim(UF)\,/\,j\in J)$ with
$\xymatrix @R=.1in{UF_j\ar[r]^-{\tau_j}\ar[d]_-{UF\alpha} &
\colim(UF)\\ UF_{j'}\ar[ur]_-{\tau_{j'}}}$ commuting for any 
$j\xrightarrow{\alpha}j'$.
Define a new diagram $G:J\to\Comod_{\colim(UF)}$ by 
$\xymatrix @C=0.15in @R=.2in {j\ar @{|->}[r]\ar[d]_-{\alpha} &
**[r](\tau_j)_*F_j=(\tau_{j'})_*(UF\alpha)_*F_j.\ar[d]^-{(\tau_{j'})_*F\alpha}\\
j'\ar @{|->}[r] & (\tau_{j'})_*F_{j'}}$
Since $\Comod_{\colim(UF)}$ is cocomplete, the above diagram
has a colimiting cocone
$((\tau_j)_* F_j\xrightarrow{\sigma_j}\colim G\,/\,j\in J)$, 
and also $U\colim G=\colim(UF)$.  

To see that $\colim G$ is actually the colimit of the initial diagram $F$,
notice that $(F_j\xrightarrow{(\sigma_j,\tau_j)} \colim G\,/\,j\in J)$
is already a cocone in $\Comod$ (the comodules being over 
the obvious coalgebras), and
$\colim G$ has the respective universal property
in $\Comod$.

In a very similar way, we can define the
category $\Mod$ of all modules over all $R$-algebras for a
commutative ring $R$. A morphism in
$\Mod$ between an $A$-module $M$ and a $B$-module $N$ 
is defined as a pair $(m,f):M_A\to N_B$ with 
$\begin{cases}
M\xrightarrow{m}f^{\sharp}N & \text{in } \Mod_A\\
A\xrightarrow{f}B & \text{in } \Alg_R
\end{cases}$, using the restriction of scalars as seen in (\ref{eq:res}).
The evident `forgetful' functor in this case is
$V:\Mod\to\Alg_R$, mapping every module $M_A$ to its
$R$-algebra $A$. Dually to the above results, we get that
$\Mod$ is a complete symmetric monoidal category, 
with a functor $F:\Mod\to\Mod_R\times\Alg_R$ 
creating all
limits, and also the functor $V$ has the structure of a 
\emph{strict} monoidal
functor, since $V(M_A\otimes N_B)=A\otimes B=VM_A\otimes VN_B$.

\begin{rmk}\label{functHom} 
There is a functor
$\Hom(-,N_B):\Comod^\mathrm{op}\to\Mod$ between the 
categories described above, induced by the
bifunctor $\Hom:\Comod_C^\mathrm{op}
\times\Mod_B\to\Mod_{\Hom_R(C,B)}$ as described in 
(\ref{eq:hommod}). More precisely, it is the partial 
functor of
\begin{equation}\label{eq:defHom}
\xymatrix @R=0.05in
{\Hom:\Comod^\mathrm{op}\times\Mod\ar[r] & \Mod\qquad\qquad\qquad\quad \\
\qquad\qquad\qquad\quad(X_C,N_B)\ar @{|->}[r] & \Hom(X,N)_{\Hom_R(C,B)}}
\end{equation}
The diagram
\begin{equation}\label{eq:cont} 
\xymatrix @R=.5in
@C=1.3in{\Comod^{\mathrm{op}}\ar[r]^-{\scriptscriptstyle{\Hom(-,N_B)}}
\ar[d]_-{G^\mathrm{op}}
& \Mod\ar[d]^-{F}\\
\Mod_R^\mathrm{op}\times\Coalg_R^\mathrm{op}
\ar[r]^-{\scriptscriptstyle{\Hom_R(-,N)\times
\Hom_R(-,B)}} & \Mod_R\times\Alg_R}
\end{equation} 
commutes, the composite on the left side 
preserves all limits (by (\ref{eq:Hcont}) and Proposition \ref{Gcomon})
and the functor $F$ creates them. Hence $\Hom(-,N_B)$ is a 
continuous functor.
\end{rmk}

\section{The existence of the universal measuring comodule}\label{existmeascomod}

Similarly to the way the universal measuring coalgebra was defined
via the adjunction (\ref{eq:meascoal}), we now 
proceed to the definition of an object $Q(M,N)_{P(A,B)}$ in
$\Comod$ 
defined by a natural isomorphism
\begin{equation}\label{eq:comod} 
\Comod(X,Q(M,N))\cong
\Mod(M,\Hom(X,N)),
\end{equation} 
where $X=X_C$, $M=M_A$ and $\Hom(X,N)=\Hom(X,N)_{\Hom_R(C,B)}$.
Hence, we want to prove the existence of an adjunction
\begin{equation}\label{eq:comodadj}
\xymatrix @C=.6in
{\Comod \ar @<+.7ex>[r]^-{\Hom(-,N_B)^\mathrm{op}} 
 _-*-<5pt>{\text{\scriptsize{$\bot$}}} &
\Mod^\mathrm{op}\ar @<+.7ex>[l]^-{Q(-,N_B)}}
\end{equation}
where $\Hom(-,N_B)^\mathrm{op}$ is the opposite of the 
functor considered in Remark \ref{functHom}. 

We begin by considering a slightly different adjunction. First of all,
from the special case of (\ref{eq:meascoal}) 
for $C=P(A,B)$, we get the
correspondence
\begin{displaymath}
\Coalg_R(P(A,B),P(A,B))
\cong\Alg_R(A,\Hom_R(P(A,B),B)),
\end{displaymath} 
and of course the identity $1_{P(A,B)}$ corresponds
uniquely to the unit of the adjunction $\alpha:A\to \Hom_R(P(A,B),B)$. 
This arrow induces, via the restriction of scalars
$\alpha^\sharp:\Mod_{\Hom_R(P(A,B),B)}\to\Mod_A$,
a set of morphisms in the category $\Mod_A$,
namely $\Mod_A(M,\alpha^{\sharp}\Hom(Z,N))$, where $M$ is an
$A$-module, $Z$ is a $P(A,B)$-module and $N$ is a $B$-module. The
question then is whether there exists a functor
$Q$ and a natural isomorphism
\begin{equation}\label{eq:specialadj}
\Comod_{P(A,B)}(Z,Q(M,N))\cong
\Mod_A(M,\alpha^{\sharp}\Hom(Z,N)).
\end{equation} 
In other words, we are looking for the left adjoint of
the functor
\begin{displaymath} 
\alpha^{\sharp}\circ
\Hom(-,N):\Comod_{P(A,B)}^\mathrm{op}\to
\Mod_{\Hom_R(P(A,B),B)}\to \Mod_A.
\end{displaymath} 
Since the category $\Comod_C$ for any $R$-coalgebra
$C$ is comonadic over $\Mod_R$, co-wellpowered and 
has a generator (see Section \ref{Modcat}), the 
Special Adjoint Functor Theorem applies:
\begin{itemize}
\item $\Comod_{P(A,B)}^\mathrm{op}$ is complete.
\item $\Comod_{P(A,B)}^\mathrm{op}$ is well-powered.
\item $\Comod_{P(A,B)}^\mathrm{op}$ has a cogenerator.
\item The functor $\alpha^{\sharp}\circ \Hom(-,N)$ is continuous,
as a composite of continuous functors.
\end{itemize}
Therefore the functor 
$\alpha^{\sharp}\circ
\Hom(-,N)$ has a left adjoint 
$Q(-,N)^{\mathrm{op}}:\Mod_A\to
\Comod_{P(A,B)}^\mathrm{op}$  and (\ref{eq:specialadj}) holds. Moreover,
the naturality of the bijection in $Z$ and $M$ makes $Q$ into 
a bifunctor 
\begin{equation}\label{eq:defQ}
Q:\Mod_A^{\mathrm{op}}\times \Mod_B\to
\Comod_{P(A,B)}
\end{equation}
such that (\ref{eq:specialadj}) is natural in all three variables.
We claim that this bifunctor is in fact the one inducing (\ref{eq:comod}).
More precisely, we are going to show that this functor $Q(-,N)$, when 
regarded as $Q(-,N_B)$, is also the left 
adjoint we are after in (\ref{eq:comodadj}).

We will establish a correspondence
between elements of $\Comod(X,Q(M,N))$ and 
elements of
$\Mod(M,\Hom(X,N))$, with $M\in\Mod_A$,
$N\in\Mod_B$ and $X\in\Comod_C$. As we saw in the
previous section, an element of the right hand side 
$(l,f):M_A\to\Hom(X,N)_{\Hom_R(C,B)}$ in $\Mod$ is 
$\begin{cases}
M\xrightarrow{l}f^\sharp\Hom(X,N) &\text{in }\Mod_A\\
A\xrightarrow{f}\Hom_R(C,B) &\text{in }\Alg_R
\end{cases}$. This arrow $l$ is a linear map $M\to
\Hom_R(X,N)$, which commutes with the coaction of $A$ on both
$R$-modules, the second becoming such via restriction of scalars along
$f$. Explicitly, $l$ satisfies the commutativity of
\begin{displaymath} 
\xymatrix @R=.05in 
{A\otimes M\ar[r]^-{1\otimes
l}\ar[dd]_-{\mu} & A\otimes \Hom(X,N)\ar
@{-->}[dd]\ar[rd]^-{f\otimes 1} &\\ 
& & \Hom_R(C,B)\otimes
\Hom(X,N)\ar[dl]^-{\mu}\\ M\ar[r]^-{l} & \Hom(X,N) &}
\end{displaymath} 
where $\Hom_R(C,B)\otimes\Hom(X,N)
\xrightarrow{\mu}\Hom(X,N)$
is the canonical action on the 
$\Hom_R(C,B)$-module $\Hom(X,N)$, as 
described in (\ref{eq:hommod}). This diagram 
translates under the
adjunction $-\otimes C\dashv \Hom_R(C,-)$ for 
the adjunct of $l$, $M\otimes X\xrightarrow{\bar{l}}N$, to
\begin{displaymath} 
\xymatrix @C=0.3in @R=0.15in 
{A\otimes M\otimes
X\ar[r]^-{1\otimes l\otimes 1}\ar[ddd]_-{\mu\otimes 1} & A\otimes
\Hom(X,N)\otimes X\ar[r]^-{f\otimes 1\otimes 1} & \Hom_R(C,B)\otimes
\Hom(X,N)\otimes X\ar[d]^-{1\otimes 1\otimes 1\otimes\delta}\\ 
& & \Hom_R(C,B)\otimes \Hom(X,N)\otimes X\otimes C\ar[d]^-s\\ 
& & \Hom_R(C,B)\otimes C\otimes \Hom(X,N)\otimes
X\ar[d]^-{e\otimes e}\\ M\otimes X\ar[dr]^-{\bar{l}} & &
B\otimes N\ar[dl]^-{\mu}\\ 
& N &}
\end{displaymath} 
using (\ref{eq:lala}). This can also be written as
\begin{equation}\label{eq:impdiag1} 
\xymatrix @C=0.2in @R=0.15in
{A\otimes M\otimes X\ar[r]^-{\scriptscriptstyle{1\otimes
l\otimes\delta}} \ar[ddd]_-{\mu\otimes 1}& 
\drtwocell<\omit>{'(*)} A\otimes \Hom(X,N)\otimes
X\otimes C\ar[r]^-{\scriptscriptstyle{f\otimes 1}}\ar @{-->}
@/_2pc/[ddr] & \Hom_R(C,B)\otimes \Hom(X,N)\otimes X\otimes
C\ar[d]^-s\\ 
&& \Hom_R(C,B)\otimes C\otimes \Hom(X,N)\otimes
X\ar[d]^-{e\otimes 1\otimes 1}\\ 
&& B\otimes \Hom(X,N)\otimes
X\ar[d]^-{1\otimes e}\\
M\otimes X\ar[dr]_-{\bar{l}} && B\otimes
N\ar[dl]^-{\mu}\\ 
& N &}
\end{equation} 
The question is whether every element of the
left hand side of (\ref{eq:comod}), when the functor $Q$ is the same as in
(\ref{eq:defQ}), corresponds uniquely to a 
linear map with the properties of $l$ as described above.
Note that, by definition of $Q$, when $M_A,N_B\in\Comod$, 
$Q(M,N)$ is a $P(A,B)$-comodule.
So, a morphism $X_C\xrightarrow{(k,h)} Q(M,N)_{P(A,B)}$ in $\Comod$ is a pair 
$\begin{cases}
h_*X\xrightarrow{k}Q(M,N) &\text{in }\Comod_{P(A,B)}\\
\quad C\xrightarrow{h}P(A,B) &\text{in }\Coalg_R
\end{cases}$. 

By the adjunction (\ref{eq:meascoal}) defining
the universal measuring coalgebra, we already know that each one of 
these coalgebra maps $h$ can be written as $\hat{f}$ for a unique 
$f:A\to \Hom_R(C,B)$ in $\Alg_R$.

The key point is that, since  
$h_*X\equiv\hat{f}_*X$ becomes a $P(A,B)$-comodule
via corestriction of scalars along $\hat{f}$, the 
map $k:\hat{f}_*X\to Q(M,N)$ in $\Comod_{P(A,B)}$
is an element of the left hand side
of the special case adjunction (\ref{eq:specialadj}), 
for $Z=\hat{f}_*X$. Therefore it uniquely corresponds to some 
$t:M\to\alpha^{\sharp}\Hom(\hat{f}_*X,N)$ in $\Mod_A$. 
We will show that this $t$, which as a
linear map is $M\xrightarrow{t}\Hom_R(X,N)$, has the property 
described by the commutativity of (\ref{eq:impdiag1}) above,
and hence it is an element of the right hand side of (\ref{eq:comod}).

We first have to see in detail how 
$\alpha^{\sharp}\Hom(\hat{f}_*X,N)$
has the structure of an $A$-module, with underlying $R$-module
$\Hom_R(X,N)$. The $A$-action is given 
by
\begin{displaymath} 
\xymatrix @R=.15in 
{A\otimes
\Hom(X,N)\ar[rr]^-{\alpha\otimes 1}\ar @{-->}@/_/[ddrr]_-{\mu'} &&
\Hom_R(P(A,B),B)\otimes \Hom(X,N)\ar[d]^-{\Hom(\hat{f},1)\otimes 1} \\
&& \Hom_R(C,B)\otimes \Hom(X,N)\ar[d]^-{\mu} \\ 
&& \Hom(X,N)}
\end{displaymath} 
which corresponds under the adjunction $-\otimes C\dashv
\Hom_R(C,-)$ to
\begin{displaymath} 
\xymatrix @C=.25in @R=.15in
{A\otimes \Hom(X,N)\otimes
X\ar[rr]^-{\alpha\otimes1\otimes1}\ar @{~>}
@/_3pc/[ddddrr]_-{\bar{\mu'}} && \Hom_R(P(A,B),B)\otimes
\Hom(X,N)\otimes X\ar[d]^-{1\otimes1\otimes\delta} \\ & &
\Hom_R(P(A,B),B)\otimes \Hom(X,N)\otimes X\otimes
C\ar[d]^-{1\otimes1\otimes1\otimes\hat{f}}\\ 
& & \Hom_R(P(A,B),B)\otimes
\Hom(X,N)\otimes X\otimes P(A,B)\ar[d]^-{1\otimes s}\\ & &
\Hom_R(P(A,B),B)\otimes P(A,B)\otimes \Hom(X,N)\otimes
X\ar[d]^{e\otimes e}\\ 
& & N.}
\end{displaymath} 
So, the regular diagram which the linear map 
$t:M\to\Hom_R(X,N)$ as an $A$-module map
has to satisfy (see (\ref{eq:defmod2}))
\begin{displaymath} 
\xymatrix {A\otimes M\ar[r]^-{1\otimes
t}\ar[d]_-{\mu} & A\otimes \Hom(X,N)\ar[d]^-{\mu'}\\ 
M\ar[r]^-t & \Hom(X,N)}
\end{displaymath} 
corresponds under $-\otimes C\dashv \Hom_R(C,-)$,
for its adjunct $\bar{t}:M\otimes X\to N$, to
\begin{equation}\label{eq:impdiag2} 
\xymatrix @C=0.05in @R=0.15in 
{A\otimes \Hom(X,N)\otimes X\otimes C
\ar[rr]^-{\alpha\otimes1}\ar
@{-->} @/_7ex/[dddrr] & \ddtwocell<\omit>{'(**)}
& \Hom(P(A,B),B)\otimes \Hom(X,N)\otimes X\otimes
C\ar[d]^-{1\otimes1\otimes1\otimes\hat{f}}\\ 
&& \Hom(P(A,B),B)\otimes
\Hom(X,N)\otimes X\otimes P(A,B)\ar[d]^-{1\otimes s}\\ 
A\otimes M\otimes X\qquad\qquad\qquad\quad
\ar @<+5ex>@/_/[uu]^-{1\otimes t\otimes\delta}
\ar @<-5ex>@/^/[dd]_-{\mu\otimes1} &&
\Hom(P(A,B),B)\otimes P(A,B)\otimes \Hom(X,N)\otimes
X\ar[d]^-{e\otimes1\otimes1}\\ 
&& B\otimes \Hom(X,N)\otimes
X\ar[d]^-{1\otimes e}\\ 
M\otimes X\quad\ar[dr]_-{\bar{t}} && B\otimes
N\ar[dl]_-{\mu}\\ 
& N &}
\end{equation}
Hence, in order for the linear map $t:M\to
\Hom_R(X,N)$ to be an element of the right hand side of
(\ref{eq:comod}), like the map $l$ described earlier, 
we just have to show that the 
diagrams (\ref{eq:impdiag1}) and
(\ref{eq:impdiag2}) are actually the same. By inspection
of the two diagrams, we see that it suffices to show
that the parts (*) and (**) are the same.
In other words, since the factor
$\Hom(X,N)\otimes X$ remains unchanged, the diagram
\begin{displaymath}
\xymatrix @C=.25in
{& \Hom_R(P(A,B),B)\otimes
C\ar[r]^-{1\otimes\hat{f}} & \Hom_R(P(A,B),B)\otimes
P(A,B)\ar[dr]^-e &\\ 
A\otimes C\ar[ur]^-{\alpha\otimes
1}\ar[drr]_-{f\otimes 1} && & B\\ 
&& \Hom_R(C,B)\otimes
C\ar[ur]_-e &}
\end{displaymath} 
must commute. But this is satisfied by Lemma \ref{lemma}, 
so the proof is complete.

To summarise, for the arbitrary element of 
$\Comod(X_C,Q(M,N)_{P(A,B)})$
that we started with, $\begin{cases}
h_*X\xrightarrow{k}Q(M,N) &\text{in }\Comod_{P(A,B)}\\
\quad C\xrightarrow{h}P(A,B) &\text{in }\Coalg_R
\end{cases}$, we already knew that $h$ corresponds uniquely
to some algebra map $f:A\to\Hom_R(C,B)$, and moreover we 
showed that $k$ corresponds uniquely to some 
$t:M\to\alpha^\sharp\Hom(\hat{f}_*X,N)$ in $\Mod_A$, which 
is a linear map $M\to\Hom_R(X,N)$ with exactly the same 
properties as an $A$-module map $M\to f^\sharp\Hom(X,N)$.
Therefore we established a (natural) correspondance
\begin{displaymath} 
\Comod(X,Q(M,N))\cong
\Mod(M,\Hom(X,N))
\end{displaymath} 
which gives the adjunction (\ref{eq:comodadj}). Notice 
that the bifunctor $Q$, defined initially as in (\ref{eq:defQ}), is used
in this context as
\begin{equation}\label{eq:defQ2}
\xymatrix @R=0.05in
{Q:\Mod^\mathrm{op}\times\Mod\ar[r] & \Comod\qquad\quad \\
\qquad\quad\quad(M_A,N_B)\ar @{|->}[r] & Q(M,N)_{P(A,B)}}
\end{equation} 
We call the object $Q(M,N)$ the \emph{universal measuring comodule}. 
As mentioned earlier, $Q(M_A,N_B)$ is a $P(A,B)$-comodule by construction.

\begin{rmk} Using the established adjunctions 
$\Hom(-,N_B)^\mathrm{op}\dashv Q(-,N_B)$ 
and $\Hom_R(-,B)^\mathrm{op}\dashv P(-,B)$,
we have a diagram of categories and functors
\begin{displaymath} 
\xymatrix @R=.5in @C=.75in 
{\Comod \ar @<+.7ex>[r]^-{\Hom(-,N_B)^\mathrm{op}}
_-*-<5pt>{\text{\scriptsize{$\bot$}}}
\ar[d]_-{U} &
\Mod^\mathrm{op}\ar @<+.7ex>[l]^-{Q(-,N_B)}
\ar[d]^-{V^\mathrm{op}} \\ 
\Coalg_R \ar @<+.7ex>[r]^-{\Hom_R(-,B)^\mathrm{op}} 
_-*-<5pt>{\text{\scriptsize{$\bot$}}} &
\Alg_R^\mathrm{op} \ar @<+.7ex>[l]^-{P(-,B)}}
\end{displaymath} 
where the squares of right adjoints/left adjoints 
respectively serially commute.
\end{rmk}

\section{Enrichment of modules in comodules}\label{enrichmmodscomods}

In a similar way to Section \ref{enrichmalgscoalgs}, we will see how the bifunctor 
$Q:\Mod^\mathrm{op}\times\Mod\to
\Comod$ induces an enrichment of the category
$\Mod$ in the symmetric monoidal category $\Comod$.
We will again use Proposition \ref{actionenrich} in the forms of Corollaries \ref{importcor1}
and \ref{importcor2},
for $\ca{C}=\Comod$ and $\ca{A}=\Mod$. 

Consider the bifunctor
$\Hom:\Comod^{\mathrm{op}}\times
\Mod\to\Mod$ as  in (\ref{eq:defHom}). In the 
previous section we showed the existence of an adjunction
$\xymatrix @C=.6in
{\Comod \ar @<+.7ex>[r]^-{\Hom(-,N_B)^{\mathrm{op}}} 
 _-*-<5pt>{\text{\scriptsize{$\bot$}}} &
\Mod^{\mathrm{op}}\ar @<+.7ex>[l]^-{Q(-,N_B)}}$. 
So again, for the conditions of Proposition {\ref{actionenrich}
to be satisfied, we only need to show that the functor $\Hom=H$ in this case 
is an action of the monoidal $\Comod^\mathrm{op}$ on $\Mod$.
We have to find two natural isomorphisms $\alpha_{XYN}:\Hom(X\otimes
Y,N)\xrightarrow{\sim}\Hom(X,\Hom(Y,N))$ and $\lambda_N:\Hom(R,N)
\xrightarrow{\sim}N$ in
$\Mod$ such that diagrams (\ref{eq:diag1}), (\ref{eq:diag2})
and (\ref{eq:diag3}) commute. We already know that these isomorphisms
exist in $\Mod_R$ as linear maps, so we have to make sure
that they can be lifted to isomorphisms in $\Mod$, as explained below.

The first isomorphism in $\Mod_R$, for $X\in\Comod_C$, 
$Y\in\Comod_D$ and $N\in\Mod_A$, is as a linear map
\begin{displaymath} 
\xymatrix @R=.05in
{k:\Hom_R(X,\Hom_R(Y,N))\ar[r] & \Hom_R(X\otimes
Y,N)\quad\\ \qquad f:X\to \Hom_R(Y,N)\ar @{|->}[r] &
[k(f)]:X\otimes Y\to N\\ & [k(f)](x\otimes y):=[f(x)](y)}
\end{displaymath}
The domain is a $\Hom_R(C,\Hom_R(D,A))$-module
and the codomain is a $\Hom_R(C\otimes D,A)$-module in 
a canonical way.
As mentioned in Section \ref{enrichmalgscoalgs}, there is an
isomorphism $\beta:\Hom_R(C,\Hom_R(D,A))\to\Hom_R(C\otimes D,A)$ in 
$\Alg_R$. Then, it can be seen that $k$ commutes with the 
$\Hom_R(C,\Hom_R(D,A))$-actions on both modules, the first one 
on $\Hom_R(X,\Hom_R(Y,N))$ being the canonical 
(see (\ref{eq:lala}))
and the second on $\Hom_R(X\otimes Y,N)$ induced via restriction
of scalars along $\beta$. 
Therefore, we obtain an isomorphism in $\Mod$, given by
$$\begin{cases}
\Hom(X,\Hom(Y,N))\xrightarrow{m}\beta^\sharp\Hom(X\otimes Y,N) 
& \text{in }\Mod_{\Hom_R(C,\Hom_R(D,A))}\\
\Hom_R(C,\Hom_R(D,A))\xrightarrow{\beta}\Hom_R(C\otimes D,A)
& \text{in }\Alg_R
\end{cases}.$$ 

With similar calculations, we obtain the second isomorphism $\lambda$
in $\Mod$, and the diagrams (\ref{eq:diag1}), (\ref{eq:diag2}) and 
(\ref{eq:diag3}) commute because they do for all $R$-modules.
Hence, the bifunctor $\Hom$ is an action, and so its opposite
$\Hom^\mathrm{op}:\Comod\times\Mod^\mathrm{op}\to\Mod^\mathrm{op}$
is also an action of the symmetric monoidal category $\Comod$ on the 
category $\Mod^\mathrm{op}$.

Now the assumptions of Corollaries \ref{importcor1}
and \ref{importcor2} hold, so we obtain
the following results.
\begin{prop}
The category $\Mod^\mathrm{op}$ is enriched in $\Comod$, with hom-objects
$\Mod^\mathrm{op}(M,N)=Q(N,M)$.
\end{prop}
\begin{prop}
The category $\Mod$ is enriched in $\Comod$, with hom-objects
$\Mod(M,N)=Q(M,N)$.
\end{prop}

\section{The category of comodules revisited}\label{catComodrevisit}

When we proved the existence of the bifunctor $Q:\Mod^\mathrm{op}\times
\Mod\to\Comod$ as the (parametrized) adjoint of $\Hom^\mathrm{op}:\Comod
\times\Mod^\mathrm{op}\to\Mod^\mathrm{op}$, we first 
considered a `special case' of this adjunction (see (\ref{eq:specialadj})) between
the specific fibers $\Comod_{P(A,B)}$ and $\Mod_A$, and then we used it to 
prove the more general adjunction between $\Comod$ and $\Mod$. Using 
a very similar idea, we can show that the symmetric monoidal category
$\Comod$ is monoidal closed. 

Recall that, by Proposition \ref{coalgmonclosed}, $\Coalg_R$ is a monoidal closed category, 
via the adjunction
\begin{equation}\label{eq:inthomc}
\xymatrix @C=.5in
{\Coalg_R\ar @<+.7ex>[r]^-{-\otimes D} 
_-*-<5pt>{\text{\scriptsize{$\bot$}}} &
\Coalg_R.\ar @<+.7ex>[l]^-{[D,-]_c}}
\end{equation}
The identity morphism $1_{[D,E]}$ corresponds uniquely to the counit of
this adjunction $\varepsilon_E:[D,E]_c\otimes D\to E$, and this arrow 
induces the corestriction of scalars 
$\varepsilon_*:\Comod_{[D,E]_c \otimes E}\to\Comod_E$ between 
the respective fibers. So we can again
consider the existence of a `special case' adjunction: for $Y\in\Comod_D$, the functor 
\begin{displaymath} 
\varepsilon_*\circ (-\otimes
Y):\Comod_{[D,E]_c}\to\Comod_{[D,E]_c\otimes E}\to
\Comod_E
\end{displaymath} 
has a right adjoint $H(Y,-):\Comod_E\to\Comod_{[D,E]_c}$, 
since again $\Comod_{[D,E]_c}$ is cocomplete, co-wellpowered, has a generator
and both $\varepsilon_*$ and $(-\otimes Y)$ preserve colimits. 
We obtain a natural isomorphism
\begin{equation}\label{eq:sth2} 
\Comod_E(\varepsilon_*(W\otimes Y),Z)\cong
\Comod_{[D,E]_c}(W,H(Y,Z))
\end{equation}
where $W\in\Comod_{[D,E]_c}$, $Y\in\Comod_D$ and $Z\in\Comod_E$.

In order to prove that $\Comod$ is monoidal closed, we are after
a more general adjunction $\xymatrix 
@C=.5in {\Comod\ar @<+.7ex>[r]^-{-\otimes Y_D} 
_-*-<5pt>{\text{\scriptsize{$\bot$}}} &
\Comod\ar @<+.7ex>[l]^-{H(Y_D,-)}}$ with a natural isomorphism
\begin{equation}\label{eq:inthomcom}
\Comod(X\otimes Y,Z)\cong
\Comod(X,H(Y,Z)).
\end{equation}
for $X\in\Comod_C$. As in Section \ref{existmeascomod}, we claim that the right adjoint $H$
of the adjunction (\ref{eq:sth2}), when seen as $H:\Comod^\mathrm{op}\times
\Comod\to\Comod$, is the same one inducing this bijection. 
So, in order to show the correspondance
between arrows 
\begin{displaymath}
\begin{cases}
f_*(X\otimes Y)\xrightarrow{k}Z &\text{in }\Comod_E\\
\quad\, C\otimes D\xrightarrow{f}E &\text{in }\Coalg_R
\end{cases}\quad
\mathrm{and}\quad 
\begin{cases}
\bar{f}_*X\xrightarrow{l}H(Y,Z) &\text{in }\Comod_{[D,E]_c}\\
\quad C\xrightarrow{\bar{f}}[D,E]_c &\text{in }\Coalg_R
\end{cases},
\end{displaymath} where $\bar{f}$ is the transpose of $f$ under (\ref{eq:inthomc}),
we can start from the right hand side and take the corresponding
\begin{displaymath}
\Big{\{}\bar{f}_*X\xrightarrow{l}H(Y,Z)\textrm{ in }\Comod_{[D,E]_c}\Big{\}}
\stackrel{(\ref{eq:sth2})}{\Longleftrightarrow}
\Big{\{}\varepsilon_*(\bar{f}_*X\otimes Y)\xrightarrow{m}Z\textrm{ in }\Comod_E\Big{\}},
\end{displaymath}
and the commutative diagram which expresses the fact that the linear map $m:X\otimes Y\to Z$
commutes with the induced $E$-action on $\varepsilon_*(\bar{f}_*X\otimes Y)$ 
turns out to be the same with the diagram expressing the fact that 
$m$ is a $E$-comodule map $f_*(X\otimes Y)\xrightarrow{m}Z$, 
using that $f$ and $\bar{f}$ are transpose maps.
Therefore (\ref{eq:inthomcom}) holds, and the bifunctor
\begin{equation}\label{eq:defH2}
\xymatrix @R=0.05in
{H:\Comod^\mathrm{op}\times\Comod\ar[r] & \Comod\qquad\quad \\
\qquad\qquad\qquad\quad(Y_D,Z_E)\ar @{|->}[r] & H(Y,Z)_{[D,E]_c}}
\end{equation}
is the internal hom of $\Comod$. Note that
$H(Y_D,Z_E)$ is a $[D,E]_c$-comodule by construction.
\begin{prop}
The global category of comodules $\Comod$ is a monoidal
closed category.
\end{prop}
The functor $-\otimes Y_D:\Comod\to\Comod$
whose right adjoint is the internal hom of $\Comod$, is 
evidently cocontinuous, by the commutativity of 
\begin{equation}\label{eq:tenY}
\xymatrix @C=.7in
{\Comod\ar[r]^-{-\otimes Y_D}\ar[d]_-{G} &
\Comod\ar[d]^-{G}\\
\Mod_R\times\Coalg_R\ar[r]^-{(-\otimes Y)\times(-\otimes D)} &
\Mod_R\otimes \Coalg_R}
\end{equation}
where $G$ is comonadic and the bottom arrow preserves colimits
since $\Mod_R$ and $\Coalg_R$ are monoidal closed.

\begin{rmk}
When $R=k$ is a field,
the existence of the universal 
measuring comodule and the internal hom in $\Comod$ become clearer. 
More precisely, rather than showing how the 
correspondences (\ref{eq:comod}), (\ref{eq:inthomcom}) are 
established explicitly, we can prove the existence of 
right adjoints of the functors $\Hom(-,N_B)^\mathrm{op}:\Comod\to
\Mod^\mathrm{op}$ and $-\otimes
Y_D:\Comod\to\Comod$ directly, using Theorem \ref{Kellythm} by Kelly.

We already know that $\Comod$ is a cocomplete category (Corollary \ref{Comodcocom}) and
the functors $\Hom(-,N_B)^\mathrm{op}$ and $-\otimes Y_D$ are cocontinuous 
(Remark \ref{functHom} and (\ref{eq:tenY})). If we can show that 
$\Comod$ has a small dense subcategory, both functors 
will automatically have right adjoints, establishing the existence of the functors $Q$ and $H$. 

Consider the category of all finite dimensional (as vector spaces)
comodules over finite dimensional coalgebras and linear maps between them, 
denoted $\ca{C}_{f.d.}$. The proof that $\ca{C}_{f.d.}$ is a 
(small) dense subcategory of
$\Comod$ makes use of the notion of the \emph{coefficients
coalgebra} Coeff($X$) over a finite dimensional $C$-comodule $X$ (see
\cite{MR2394705} as the space of matrix coefficients), which is the
smallest subcoalgebra of $C$ such that $X$ is a Coeff($X$)-comodule,
and even more it is a finite dimensional coalgebra.

\begin{proof}[Sketch of proof] For a fixed $X_C\in\Comod$, consider the category
\begin{displaymath} 
\Lambda=\{(V,D)\;\epsilon\;Sub_f(X)\times
Sub_f(C)/\; \mathrm{Coeff}(V)\subset D\}
\end{displaymath} 
where $Sub_f(C)$ is the preorder of finite
dimensional subcoalgebras of $C$ and $Sub_f(X)$ the
preorder of finite dimensional subcomodules of the comodule $X$. 
For each $(V,D)\in\Lambda$, since $V$
is a subcomodule of $X_C$, $V\in\Comod_{C}$, but we can
restrict the coaction to $D\subset C$ and we denote the $D$-comodule
$V_D$.  Then, we can define a diagram
\begin{displaymath} 
\xymatrix @R=.03in{\phi:\quad\;\Lambda\ar[r] &
\Comod\\ 
\quad (V,D)\ar @{|->}[r] & V_D\qquad}
\end{displaymath} 
and we have a cocone
$(V_D\xrightarrow{(\tau_{(V,D)},g)}X_C\,/\,(V,D)\in\Lambda)$
in $\Comod$, where $g$ is the
inclusion $g=\iota_D: D\hookrightarrow C$ in $\Coalg_k$ and
$\tau_{(V,D)}:(\iota_D)_*V_D\to X$ is again the inclusion
$j_V:V\hookrightarrow X$ in $\Comod_C$.

The Fundamental Theorem of Coalgebras and the
Fundamental Theorem of Comodules state that $C$ is the
colimit of its f.d. subcoalgebras and $X$ is the colimit of its
f.d. subcomodules, and then we can show that both $g$ and
$\tau_{(V,D)}$ are colimiting cocones in $\Coalg_k$ and
$\Comod_C$ respectively. Therefore
$(V_D\xrightarrow{(\tau_{(V,D)},g)}X_C\,/\,(V,D)\in\Lambda)$
is a colimiting cocone in $\Comod$.

In order to prove that $\ca{C}_{f.d}$ is dense in
$\Comod$, it suffices to show that the diagram
\begin{displaymath} 
\xymatrix {(I\downarrow X_C)\ar[r]^-P &
\ca{C}_{f.d}\; \ar @{^{(}->} [r]^-I & \Comod}
\end{displaymath} has $X_C$ as its canonical colimit, for any
$X_C\in\Comod$. For that, we form the diagram
\begin{displaymath} 
\xymatrix @R=.3in @C=.7in
{\Lambda\ar[r]^-F\ar[ddr]_-\phi & (I\downarrow X_C)\ar[d]_-{P}\\ 
& \ca{C}_{f.d.}\ar[d]^-I\\ & \Comod}
\end{displaymath} 
which commutes, and we can prove that the functor $F$ is
final, using epi-mono factorizations and finiteness properties of
coalgebras and comodules. We already know the colimiting 
cocone for the diagram $\phi$, therefore every comodule is the canonical
colimit of elements of $\ca{C}_{f.d.}$ and the proof is complete.
\end{proof}
\end{rmk}

\bibliographystyle{plain} \bibliography{myreferences}

\end{document}